\documentclass[12pt]{amsart}
\usepackage{amssymb,amsmath}
\usepackage{tikz}
\usepackage{tikz-cd}
\def\dint{\displaystyle\int}
\def\dsum{\displaystyle\sum}
\def\dprod{\displaystyle\prod}
\def\t{\mathbf{t}}
\def\x{\mathbf{x}}
\def\<{\left<}
\def\>{\right>}
\def\M{\bar{\mathcal{M}}}

\def\S{\mathcal{S}}

\newtheorem{thm}{Theorem}

\newtheorem{lemma}{Lemma}[section]
\newtheorem{lemma*}{Lemma}

\newtheorem{prop}{Proposition}[section]
\newtheorem{prop*}{Proposition}

\begin{document}
\title{A formula on $g=1$ quantum K-invariants}
\author{Dun Tang}

\begin{abstract}
In this paper, we prove an analog of Dijkgraaf-Witten's theorem for $g=1$ invariants in quantum K-theory.
\end{abstract}

\maketitle

\tableofcontents

\section*{Introduction}
\

Let $X$ be a compact Kahler manifold, and $\M_{g,n}(X,d)$ be the compactified moduli space of stable maps of degree $d \in H^2(X,\mathbb{Z})$ to $X$.
The moduli space $\M_{g,n}(X,d)$ carries a virtual structural sheaf $\mathcal{O}^{vir}_{g,n,d}$ defined by Lee \cite{Lee01}.
For each $i \in \{1,2,\cdots,n\}$, there is a universal cotangent line bundles $L_i$ over $\M_{g,n}(X,d)$ obtained by gluing together cotangent lines at the $i^{th}$ marked point of the base curve, and an evaluation map $ev_i: \M_{g,n}(X,d) \to X$ defined by the image of the $i^{th}$ marked point under the stable map.
Let $K=K^0(X) \otimes \Lambda$, where $\Lambda$ is the Novikov ring that contains Novikov's variables $Q$'s.
Given $\t_i(q) \in K[q^\pm]$, we define $\Lambda$-valued correlators
\[\<\t_1,\cdots, \t_n\>_{g,n,d} = \chi(\M_{g,n}(X,\beta), \mathcal{O}^{vir}_{g,n,d} \otimes \left(\otimes_{i=1}^n (ev_i^\ast\t_i)(L_i)\right)).\]

We compute generating functions
\[\mathcal{F}_1(\t) = \sum_{n,d} \frac{Q^d}{n!} \<\t,\cdots, \t\>_{1,n,d}.\]
Note that from the computational point of view, Y.P. Lee and F. Qu gave an effective algorithm of $\mathcal{F}_1(\t)$ for the point target space \cite{Lee12}.
Y.C. Chou, L. Herr and Y.P. Lee \cite{Chou23} also constructed genus $g$ quantum K invariants of a general target $X$, from permutation equivariant $g=0$ invariants of the orbifold $Sym^{g+1}X = X^{g+1}/S_{g+1}$, where $S_{g+1}$ is the permutation group of $g+1$ elements, that permutes the $g+1$ copies of $X$.
In this paper, we gave a reconstruction formula that expresses $\mathcal{F}_1(\t)$ (for general targets $X$) in terms of $g=0$ invariants and $g=1$ invariants of $X$ that do not involve universal cotangent bundles at all but the first marked point.
Our main theorem is formulated similarly to Dijkgraaf-Witten's theorem on $g=1$ cohomological Gromov-Witten invariants.

\begin{thm}
The no-permutation invariant is
\[\mathcal{F}_1(\t) = F_1(\tau) + \frac1{24} \log \det\left(\frac{\partial \tau^\beta}{\partial t_0^\alpha}\right)+F^{tw}_1(\bar{\t}^{new}) - F^{tw}_1(\bar{\t}^{fake}).\]
Where $F_1(\tau) = \mathcal{F}_1(\tau)$ is the $g=1$ primary invariants.
\end{thm}

Ingredients other than $F_1$ in this theorem are defined as follows, in terms of $g=0$ invariants.
\begin{enumerate}
\setcounter{enumi}{-1}
\item
Given $\tau \in K$, define generating functions
\[\<\!\<\t_1,\cdots, \t_n\>\!\>_{g,n} = \sum_{d,\bar{n}} \frac{Q^d}{\bar{n}!}\<\t_1,\cdots, \t_n; \tau, \cdots, \tau\>_{g,n+\bar{n},d}.\]
Let $\{\phi_\alpha\}$ be a basis of $K(X) \otimes_\mathbb{Z} \mathbb{Q}$, and $(\cdot,\cdot)$ be the Poincare pairing on $K(X)$ defined by $(\phi,\psi) = \chi(X, \phi\otimes \psi)$.
Define
\[G_{\alpha\beta}= (\phi_\alpha,\phi_\beta)+ \<\!\<\phi_\alpha, \phi_\beta\>\!\>_{0,2}\]
and let $G^{\alpha\beta}$ be so that the matrix $\left[G^{\alpha\beta}\right] = \left[G_{\alpha\beta}\right]^{-1}$.
We define
\[\S_\tau(\phi)(q) = \sum_{\alpha,\beta} \left((\phi,\phi_\alpha)+\left<\!\left< \frac{\phi}{1-L/q}, \phi_\alpha\right>\!\right>_{0,2}\right)G^{\alpha\beta}\phi_\beta.
\]
The operator $\S_\tau:K \to K \otimes \mathbb{Q}((q))$ extends $\mathbb{Q}((q))$-linearly to an operator defined on $K \otimes \mathbb{Q}((q))$, which we also denote by $\S_\tau$.
\item
Given $\t \in K[q^\pm]$ and $\tau \in K$, we define $\bar{\t}$ by $(1-q+\bar{\t}) = [\mathcal{S}_\tau(1-q+\t)]_+$, where $[\cdots]_+$ stands for extracting the Laurent polynomial part of a rational function.
We choose $\tau$ so that $\bar{\t}(1)=0$.
This defines $\bar{\t}$ and $\tau$.
\item
Let $t_0 = \t(1)$.
Expand $\tau,t_0$ as $\tau=\sum_\beta \tau^\beta \phi_\beta, t_0 = \sum_\alpha t_0^\alpha \phi_\alpha$.
This defines $\tau^\beta$ and $t_0^\alpha$.
\item
Let $\phi^\alpha$ be the basis of $K(X) \otimes \mathbb{Q}$ dual to $\phi_\alpha$ under the pairing $(\cdot,\cdot)$.
Let
\[\begin{array}{ll}
\bar{\t}^{fake}(q) &= \tau + \dsum_\alpha \phi^\alpha \<\!\<\dfrac{\phi_\alpha}{1-qL_1}\>\!\>_{0,1} \\
\bar{\t}^{new}(q) &= \bar{\t}(q) + \dsum_\alpha \phi^\alpha \<\!\<\dfrac{\phi_\alpha}{1-qL_1},\bar{\t}(q)\>\!\>_{0,2} + \bar{\t}^{fake}(q)
\end{array}\]
\item
Define
\[F^{tw}_1(\bar{\t}) = \sum_{a=0,1,\infty}Res_a \left(\sum_{n=0}^3 \sum_d \frac{Q^d}{(n+1)!} \<\frac{\bar{\t}(x^{-1})}{1-xL_1},\bar{\t}(x^{-1}),\cdots,\bar{\t}(x^{-1})\>_{1,n+1,d} \frac{dx}{x}\right).\]
\end{enumerate}

This paper is organized as follows.
In the first section, we prove this theorem.
In the second section, we verify that our theorem and the result in \cite{Lee12} agree for $2$-point invariants $\<\frac1{1-q_1 L}, \frac1{1-q_2L}\>_{1,2}$ of the point target space.
In the third section, we prove (in Theorem 2) that for the point target space,
\[\sum_{d,n} \frac{Q^d}{n!} \<\frac{1}{1-qL}, \t(L), \cdots, \t(L) \>_{1,1+n,d}\]
could be reconstructed from $\<\!\<\frac{1}{1-qL}\>\!\>_{1,1}$, and (in Proposition 3.1) that $\<\!\<\frac{1}{1-qL}\>\!\>_{1,1}$ could be reconstructed from $\<\frac{1}{1-qL}\>_{1,1}$ for the point target space.
In the appendix, we give an algorithm that computes $\tau$ to arbitrary accuracy.

\section*{Acknowledgments}
\

The author thanks Professor Alexander B. Givental for drawing his attention to $g=1$ quantum K-theory, suggesting related problems, and providing patience and guidance throughout the process.

\section{Proof of Theorem 1}
\

We define the matrix 
\[A = \left[\sum_\gamma G^{\alpha\gamma} \<\!\<\phi_\gamma, \bar{t}_1, \phi_\beta\>\!\>_{0,3}\right]\]
whose rows and columns are labeled by $\alpha,\beta$.

\begin{prop}
\[(I-A)^{-1} = \left[\frac{\partial \tau^\beta}{\partial t_0^\alpha}\right].\]
\end{prop}

We first prove the following result on $[\mathcal{S}_\tau \x]_+$ expanded at $q=1$, for $\x \in K^0(X) \otimes \Lambda[q^\pm]$.

Remark that the $\frac1{1-L/q}$ in $\S_\tau$ is understood as follows.
There is a Laurent polynomial $P$ with zeros at roots of unities that annihilates $L$.
So $\Phi(q_1,q_2) = \frac{P(q_1)-P(q_2)}{q_1-q_2}$ is a Laurent polynomial, and we interpret $\frac1{1-L/q}$ as the Laurent polynomial $\frac{q\Phi(q,L)}{P(q)}$ in $L$.

\begin{lemma}
Expanded as a power series in $q-1$,
\begin{enumerate}
\item
The degree $0$ term in $[\mathcal{S}_\tau \x]_+$ is
\[\sum_{\alpha,\beta}\phi_\alpha G^{\alpha\beta} \<\!\< \phi_\beta,\x(L),1\>\!\>_{0,3}.\]
\item
If the degree $0$ term in $[\mathcal{S}_\tau \x]_+$ vanish, then the degree $1$ term is
\[\sum_{\alpha,\beta}\phi_\alpha G^{\alpha\beta} \<\!\< \phi_\beta,\x(L),1,1\>\!\>_{0,4}.\]
\end{enumerate}
\end{lemma}

\begin{proof}
Let $f(q)$ be a rational function with zero Laurent polynomial part, and let $g(q)$ be a Laurent polynomial.
A direct computation shows that the Taylor expansion of $[fg]_+$ at $1$, is the power series part of $fg$ with $f$ expanded as a power series in $(q-1)^{-1}$ and $g$ expanded as a power series in $(q-1)$.

If $\x(q) = \sum_{k \geq 0} x_k(q-1)^k$, then by
\[\begin{array}{ll}
\dfrac1{1-L/q} & = \dsum_{m \geq 0} \dfrac{(L-1)^m q}{(q-1)^{m+1}}\\
&= \dsum_{m \geq 0} (L-1)^m (\dfrac1{(q-1)^{m+1}} + \dfrac1{(q-1)^m})
\end{array}\]
we have
\[\begin{array}{ll}
\left[\dfrac{\x(q)}{1-L/q} \right]_+& = \dsum_{k \geq m \geq 0} x_k (L-1)^m((q-1)^{k-m} + (q-1)^{k-m-1})\\
&= \dsum_n (q-1)^n \left[\dfrac{\x(L)}{(L-1)^n} + \dfrac{\x(L)}{(L-1)^{n+1}}\right]_+.
\end{array}\]

Thus the degree $0$ term of $\left[\frac{\x(q)}{1-L/q} \right]_+$ when expanded as a power series in $q-1$ is
\[\left[\x(L)+\dfrac{\x(L)}{L-1}\right]_+ = \x(L)+\dfrac{\x(L)-\x(1)}{L-1}.\]
We get the first part from this and the String equation \cite{perm7}.

As for the second part, by String equation \cite{perm7} we have
\[\<\!\< \phi_\beta,\x(L),1,1\>\!\>_{0,4} = \<\!\< \phi_\beta,\x(L),1\>\!\>_{0,3} + \<\!\< \phi_\beta,\frac{\x(L)-\x(1)}{L-1},1\>\!\>_{0,3}.\]
The first term on the right-hand side vanishes by assumption.
Since
\[\left[\left(1+\dfrac1{L-1}\right)\left(\dfrac{\x(L)-\x(1)}{L-1}\right)\right]_+ = \left[\left(\dfrac1{L-1} + \dfrac1{(L-1)^2}\right)\x(L)\right]_+,\]
the second part of the lemma follows from the first part with $\frac{\x(L)-\x(1)}{L-1}$ in the place of $\x(L)$, 
\end{proof}

\begin{proof}
(Proof of Proposition 1.1)
Use $\x = 1-q+\t$.
By the previous lemma and the assumption that $\bar{\x}(1)=0$ we have
\[\sum_\gamma G^{\alpha\gamma}\<\!\<\phi_\gamma, 1- \bar{t}_1, \phi_\beta\>\!\>_{0,3}
=\sum_{\gamma,\delta,\epsilon} G^{\alpha\gamma}\<\!\<\phi_\gamma, \phi_\delta, \phi_\beta\>\!\>_{0,3} G^{\delta\epsilon}\<\!\<\phi_\epsilon, 1,1,\x\>\!\>_{0,4}.\]
By WDVV equation and the fact that $G_{\gamma\delta} = \<\!\<1, \phi_\gamma,\phi_\delta\>\!\>_{0,3}$, the right-hand side of the previous equation is simplified as
\[\sum_{\gamma,\delta,\epsilon} G^{\alpha\gamma}\<\!\<\phi_\gamma, \phi_\delta, 1\>\!\>_{0,3} G^{\delta\epsilon}\<\!\<\phi_\epsilon, \phi_\beta,1,\x\>\!\>_{0,4} 
= \sum_\epsilon G^{\alpha\epsilon}\<\!\<\phi_\epsilon, \phi_\beta,1,\x\>\!\>_{0,4}.\]

On the other hand, the first part of the lemma says
\[0=\<\!\<\phi_\beta,1,\x\>\!\>_{0,3}.\]
Note that $\tau$ is computed from $\x$.
We write $x (=\x(1))= \sum_\alpha x^\alpha \phi_\alpha$ and $\tau = \sum_\beta \tau^\beta \phi_\beta$.
Taking partial derivative in $x^\alpha$, we have
\[\begin{array}{ll}
0 &= \<\!\< \phi_\beta, 1, \phi_\alpha \>\!\>_{0,3_1} + \dsum_\gamma \<\!\< \phi_\beta, 1, \x, \phi_\gamma \>\!\>_{0,4} \dfrac{\partial \tau^\gamma}{\partial x^\beta} \\
&= G_{\alpha\beta} + \dsum_\gamma \<\!\< \phi_\beta, 1, \x, \phi_\gamma \>\!\>_{0,4} \dfrac{\partial \tau^\gamma}{\partial x^\beta}.
\end{array}\]
Thus the matrix 
\[I-A =\sum_\gamma G^{\alpha\gamma}\<\!\< \phi_\gamma,1-\bar{t}_1,\phi_\beta\>\!\>_{0,3}\]
is inverse to $\frac{\partial \tau^\beta}{\partial x^\alpha}$.
Since $x$ coordinates and $t$ coordinates differ by only a shift, taking partial derivatives in $x^\alpha$ is the same as taking partial derivatives in $t^\alpha$.
\end{proof}

Going back to the proof of Theorem 1, the first term comes from the Ancestor-Descendant correspondence \cite{perm7}
\[\mathcal{F}_1(\t) = F_1(\tau) + \bar{\mathcal{F}}_1(\bar{\t}).\]

By definition,
\[\bar{\mathcal{F}}_1(\bar{\t}) = \sum_{n,\bar{n},d}\frac{Q^d}{n! \bar{n}!} \cdot \chi\left(\M_{1,n+\bar{n}}(X,d),\prod_{i=1}^n ev_i^\ast\bar{\t}(\bar{L}_i) \cdot \prod_{j=n+1}^{n+\bar{n}} ev_j^\ast\tau\right),\]
where $\bar{L}_i$ is the pull back of $L_i$ along the forgetful maps $\pi_{\bar{n},n}: \cup_d \M_{1,n+\bar{n}}(X,d) \to \M_{1,n}$ that forgets the last $\bar{n}$ marked points and the map to $X$.
We simplify this holomorphic Euler characteristic by using the virtual Kawasaki Riemann-Roch theorem \cite{Tonita14} and pushing forward along $\pi_{\bar{n},n}$.
Define the fake Euler characteristics by
\[\chi^{fake}(\mathcal{M},V) = \int_{[\mathcal{M}]^{vir}} ch(V)td(X).\]
Then Kawasaki's Riemann Roch theorem states
\[\chi(\mathcal{M},V) = \chi^{fake}\left(I\mathcal{M},\frac{V}{\wedge^\ast N^\ast_{I\mathcal{M}/\mathcal{M}}}\right),\]
where $N_{I\mathcal{M}/\mathcal{M}}$ is the normal bundle of $I\mathcal{M}$ in $\mathcal{M}$.
\

The second term in the theorem is the contribution from non-twisted strata.
By dimension arguments, for each $n, \bar{n}$, the fake Euler characteristic
\[\begin{array}{ll}
&\chi^{fake}\left(\M_{1,n+\bar{n}}(X,d),\dprod_{i=1}^n ev_i^\ast\bar{\t}(\bar{L}_i) \cdot \dprod_{j=n+1}^{n+\bar{n}} ev_j^\ast\tau\right)\\
=&\dint_{\M_{1,n}} ch\left(\dprod_{i=1}^n(L_i-1)\right) \cdot \int_{\pi_{\bar{n}, n}^{-1} (\{C\})} \left(ch\big(\dprod_{i=1}^n ev_i^\ast\bar{t}_1 \cdot \dprod_{j=n+1}^{n+\bar{n}} ev_j^\ast\tau\big) td \big(\pi_{\bar{n}, n}^{-1} (\{C\}) \big)\right) ,
\end{array}\]
where $\bar{t}_1 = \bar{\t}'(1)$, the coefficient of the $(q-1)$-term in $\bar{\t}(q)$.
Here 
\[\int_{\pi_{\bar{n}, n}^{-1} (\{C\})} \left(ch\big(\dprod_{i=1}^n ev_i^\ast\bar{t}_1 \cdot \dprod_{j=n+1}^{n+\bar{n}} ev_j^\ast\tau\big) td \big(\pi_{\bar{n}, n}^{-1} (\{C\}) \big)\right)\]
is the (orbi-)bundle over $\M_{1,n}$, with fiber over each point $\{C\} \in \M_{1,n}$ given by the integral.
By dimension argument, the fake Euler characteristics only depend on the zeroth Chern class of this bundle.
Also, notice that the isotropy of fibers are subgroups of the isotropy of base points. By applying the (virtual) Hirzebruch-Riemann-Roch theorem, this integration is written as 
\[(\pi_{\bar{n}, n})_\ast (\dprod_{i=1}^n ev_i^\ast\bar{t}_1 \cdot \dprod_{j=n+1}^{n+\bar{n}} ev_j^\ast\tau),\]
with virtual dimension 
\[\left((\pi_{\bar{n}, n})_\ast (\dprod_{i=1}^n ev_i^\ast\bar{t}_1 \cdot \dprod_{j=n+1}^{n+\bar{n}} ev_j^\ast\tau)\right)\bigg|_{C_n},\]
for any curve $C_n$ with trivial isotropy.
We pick $C_n$ as the curve given by a cycle of $n$ spheres connected along the north and south poles, each with a marked point.
By \cite{WDVV}, the effect of connecting two spheres along a node that represents points $p,p'$ under normalization, is using $\sum_{\alpha,\beta} G^{\alpha\beta} \phi_\alpha \otimes \phi_\beta$ as the input for $p,p'$.
Again, by dimension argument and the cohomological dilaton equation, the integral
\[\dint_{\M_{1,n}} ch\left(\dprod_{i=1}^n(L_i-1)\right) = \dint_{\M_{1,n}} \dprod_{i=1}^n c_1(L_i) = \dfrac{(n-1)!}{24} \]
Take into account the coefficient $\frac{Q^d}{n!\bar{n}!}$, sum over $n,\bar{n},d$, and use Proposition 1.1, we get
\[\sum_n\frac1{24n} tr(A^n) = \frac1{24} tr \log (I-A)^{-1} = \frac1{24} tr\log \left(\frac{\partial \tau^\beta}{\partial t_0^\alpha}\right) = \frac1{24} \log \det \left(\frac{\partial \tau^\beta}{\partial t_0^\alpha}\right)\]
as the contribution from the non-twisted stratum to $\bar{\mathcal{F}}_1(\bar{\t})$.

The last two terms in the theorem come from twisted strata.
Indeed, by the residue theorem and the fact that $\<\frac{\phi_\alpha}{1-xL_1},\tau,\cdots,\tau\>_{1,n+1,d}$ only has poles at roots of unities, we have
\[\begin{array}{ll}
F^{tw}_1(\bar{\t}) &= \dsum_{a=0,1,\infty} Res_a \left(\dsum_d \dsum_{n=0}^3 \dfrac{Q^d}{(n+1)!} \<\dfrac{\bar{\t}(x^{-1})}{1-xL_1},\bar{\t}(x^{-1}),\cdots,\bar{\t}(x^{-1})\>_{1,n+1,d} \dfrac{dx}{x}\right)\\
&=-\dsum Res_\zeta\left(\dsum_d \dsum_{n=0}^3 \dfrac{Q^d}{(n+1)!} \<\dfrac{\bar{\t}(x^{-1})}{1-xL_1},\bar{\t}(x^{-1}),\cdots,\bar{\t}(x^{-1})\>_{1,n+1,d} \dfrac{dx}{x}\right)\\
&=\dsum Res_\zeta\left(\dsum_d \dsum_{n=0}^3 \dfrac{Q^d}{(n+1)!} \<\dfrac{\bar{\t}(x)}{1-x^{-1}L_1},\bar{\t}(x),\cdots,\bar{\t}(x)\>_{1,n+1,d} \dfrac{dx}{x}\right),
\end{array}\]
where the latter two sums are over roots of unities $\zeta \neq 1$.

Now from the expansion
\[\frac1{x-L_1} = \sum_n \frac{(L_1-\zeta)^k}{(x-\zeta)^{k+1}}\]
we know that
\[\begin{array}{ll}
&Res_\zeta \<\dfrac{\bar{\t}^{new}(x)}{1-x^{-1}L_1},\bar{\t}^{new}(x),\cdots,\bar{\t}^{new}(x)\>_{1,n+1,d} \dfrac{dx}{x}\\
=& \<\bar{\t}^{new}(\zeta),\bar{\t}^{new}(\zeta),\cdots,\bar{\t}^{new}(\zeta)\>_{1,n+1,d}\\
+& (n+1) \<\bar{\t}^{new}(L_1)-\bar{\t}^{new}(\zeta),\bar{\t}^{new}(\zeta),\cdots,\bar{\t}^{new}(\zeta)\>_{1,n+1,d}
\end{array}\]
On the other hand, by dimension argument, no (rational) components will bubble out from fixed points of the symmetry of twisted strata after forgetting all marked points with input $\tau$'s.
So all marked points carrying input $\tau$'s are located on a tree of rational components emerging from a fixed point of the symmetry of the twisted strata.
Recall that
\[J(\tau) = 1-q+\tau+\sum_{n,d,\alpha,\beta}\frac{Q^d}{n!}\phi_\alpha g^{\alpha\beta}\<\frac{\phi_\beta}{1-qL_1}, \tau,\cdots,\tau\>_{0,1+n,d}.\]
\begin{enumerate}
\item
The contribution from such trees, with one no-$\bar{L}$ input $\bar{t}$ and arbitrarily many inputs $\tau$'s, is given by the Lie derivative of $J(\tau)$ in the direction $\bar{t}$.
By $\mathbb{Q}[\bar{L}]$-linearity, if the input is a polynomial $\bar{\t}(\bar{L})$ instead of $\bar{t}$, the contribution is the Lie derivative of $J(\tau)$ in the direction $\bar{\t}(x)$, then with the place-holder $x$ replaced by $L_1(=\bar{L}_1)$.
\item
The contribution from such trees with all input $\tau$'s is $J(\tau)-(1-q)$.
\end{enumerate}

Putting things together, we know that the inputs add up to $\bar{\t}^{new}(q)$ as desired, and the correlators sum up to
\[\<\bar{\t}^{new}(\zeta),\bar{\t}^{new}(\zeta),\cdots,\bar{\t}^{new}(\zeta)\>_{1,n+1,d}+ (n+1) \<\bar{\t}^{new}(L_1)-\bar{\t}^{new}(\zeta),\bar{\t}^{new}(\zeta),\cdots,\bar{\t}^{new}(\zeta)\>_{1,n+1,d}\]

Notice that there is an over-count: when all marked points carry input $\tau$'s, which contributes $F^{tw}_1(\bar{\t}^{fake})$.
This over-count accounts for the last term.

\section{Example: $\<\frac1{1-q_1 L}, \frac1{1-q_2L}\>_{1,2}$ of point target space}
\

The algorithm in \cite{Lee12} gives two point invariants for the point target space
\[\begin{array}{ll}
&\<\dfrac1{1-q_1 L}, \dfrac1{1-q_2L}\>_{1,2}\\
=& \dfrac1{(1-q_1)(1-q_2^2)(1-q_2^3)(1-q_2^4)}+\dfrac1{(1-q_2)(1-q_1^2)(1-q_1^3)(1-q_1^4)}-\dfrac1{(1-q_1)(1-q_2)}\\
+&\dfrac1{24}\dfrac{q_1q_2}{(1-q_1)^2(1-q_2)^2}\\
+&\dfrac18\dfrac{q_1q_2}{(1-q_1^2)(1-q_2^2)}\left(11-\dfrac{2q_1}{1+q_1}-\dfrac{2q_2}{1+q_2}\right)\\
+&\dfrac14\dfrac{q_1q_2}{(1-q_1)(1-q_2)}\dfrac{1+(q_1+q_2)-q_1q_2}{(1+q_1^2)(1+q_2^2)}\\
+&\dfrac13\dfrac{q_1q_2}{(1-q_1)(1-q_2)}\dfrac{1+2(q_1+q_2)+q_1q_2}{(1+q_1+q_1^2)(1+q_2+q_2^2)}
\end{array}\]

We re-prove this result with Theorem 1 and the facts
\[\begin{array}{ll}
\<\dfrac1{1-qL}\>_{1,1} & =\dfrac1{(1-q^4)(1-q^6)}\\
\<\dfrac1{1-qL},1\>_{1,2} & =\dfrac1{(1-q^2)(1-q^3)(1-q^4)}
\end{array}\]

\subsection{Preparations}
\

When $X=pt$, it is known that
\[1-q+\tau+\<\!\<\frac1{1-qL}\>\!\>_{1,1}=(1-q) e^{\frac{\tau}{1-q}}.\]
Thus we have
\[\begin{array}{lll}
\bar{\t}^{new} (q) &= (\bar{\t} (q)+1-q) & \cdot e^{\tau/(1-q)} -(1-q),\\
\bar{\t}^{fake} (q) &= (1-q) & \cdot e^{\tau/(1-q)} -(1-q),
\end{array}\]

Note that
\[\<\frac1{1-q_1 L}, \frac1{1-q_2L}\>_{1,2} = \left(D_1D_2\mathcal{F}_1(\t)\right)|_{\t=0},\]
where $D_i = \sum_{n\geq 0} \frac{q_i^n}{(1-q_i)^{n+1}}\partial_{t_n}$.
Let $I=(t_0,\cdots,t_n,\cdots)$.

We have
\[\tau = t_0+t_0t_1+\frac{t_0^2t_2}2+t_1^2t_0 + O(I^4).\]
Thus modulo $I^3$ we have
\[\begin{array}{ll}
\bar{\t}(q) &= \left[e^{\frac{\tau}{q-1}}(\t(q)+1-q)\right]_+-(1-q)\\
&=\t(q)+\dfrac{\tau(\t(q)-t_0)}{q-1}-\tau,
\end{array}\]
and
\[\begin{array}{ll}
\bar{\t}^{new}(q) &= e^{\frac{\tau}{1-q}}(\bar{\t}(q)+1-q)-(1-q)\\
&=\t(q)+\dfrac{t_0^2}{2(1-q)}.
\end{array}\]
We also have
\[\begin{array}{ll}
\bar{\t}^{fake}(q) &= e^{\frac{\tau}{1-q}}(1-q)-(1-q)\\
&=\tau+\dfrac{\tau^2}{2(1-q)}\\
&=t_0+t_0t_1+\dfrac{t_0^2}{2(1-q)}.
\end{array}\]

We need the following partial fraction decompositions.

\[\begin{array}{ll}
\dfrac1{q(1-q^{-2})(1-q^{-3})(1-q^{-4})}&=\dfrac{9q+7}{32(q+1)^2}+\dfrac{q-1}{8(q^2+1)}+\dfrac{2q+1}{9(q^2+q+1)}+O\left(\dfrac1{1-q}\right);\\
\dfrac1{q(1-q^{-4})(1-q^{-6})}&=\dfrac1{24(q-1)^2}+\dfrac5{24(q-1)}+O((q-1)^0);
\end{array}\]
And we have
\[\frac1{(1+q)(1-q^3)(1-q^4)} = \frac{3q+4}{8(q+1)^2}-\frac{q}{4(q^2+1)}+\frac1{3(q^2+q+1)}+\frac1{24(1-q)^2}+\frac1{8(1-q)}.\]

\subsection{Computations via Theorem 1}
\

We shorthand $a_i=\frac{q_i}{1-q_i}$ for $i=1,2$.
Thus $D_i=\dfrac1{1-q_i}\dsum_{n\geq 0} a_i^n \partial_{t_n}$.

\subsubsection{Contributions from $F_1(\tau)$}
\

Modulo $I^3$, the first term in Theorem 1 is
\[\begin{array}{ll}
F_1(\tau)&=\<\tau\>_{1,1}+\dfrac{\<\tau,\tau\>_{1,2}}2\\
&=\tau+\dfrac{\tau^2}2\\
&=(t_0+t_0t_1)+\dfrac{t_0^2}2,
\end{array}\]
thus it contribute (to $\<\frac1{1-q_1 L}, \frac1{1-q_2L}\>_{1,2}$)
\[\frac1{(1-q_1)(1-q_2)}(a_1+a_2+1).\]

\subsubsection{Contributions from $\frac1{24}\log \frac{\partial \tau}{\partial t_0}$}
\

Modulo $I^3$, we have
\[\dfrac{\partial \tau}{\partial t_0}= 1+t_1+t_0t_2+t_1^2.\]
Thus
\[\frac1{24}\log \dfrac{\partial \tau}{\partial t_0} = \frac1{24}\left((t_1+t_0t_2+t_1^2)-\frac{t_1^2}2\right),\]
and this part contribute
\[\frac1{24}\frac1{(1-q_1)(1-q_2)}(a_1^2+a_1a_2+a_2^2).\]

\subsubsection{Contributions from $n=0$ term in $F^{tw}_1(\bar{\t}^{new})-F^{tw}_1(\bar{\t}^{fake})$}
\

Modulo $I^3$ and substitute $q=x^{-1}$, what we need to compute is
\[\left(Res_0+Res_1+Res_\infty\right)\dfrac{t_0t_1}{q(1-q^{-4})(1-q^{-6})}dq.\]
The residue at $0$ is $0$.
The residue at $\infty$ is $-1$.
By the second partial fraction decomposition in Section 2.1, the residue at $1$ is $\frac5{24}$.
Thus, this part contribute 
\[-\frac{19}{24}\frac1{(1-q_1)(1-q_2)}(a_1+a_2).\]

\subsubsection{Contributions from $n=1$ term in $F^{tw}_1(\bar{\t}^{new})$}
\

Use the residue formula and substitute $q=x^{-1}$, what we need to compute is
\[\sum_{\zeta = -1,\pm i,\omega^{\pm 2}} Res_\zeta \frac12 (\bar{\t}^{new}(q))^2 \dfrac{dq}{q(1-q^{-2})(1-q^{-3})(1-q^{-4})}.\]
Here $\omega = e^{\frac{2\pi i}6}$.
Applying $D_1D_2$ yields
\[\sum_{\zeta = -1,\pm I,\omega^{\pm 2}} Res_\zeta \dfrac1{(1-qq_1)(1-qq_2)} \dfrac{dq}{q(1-q^{-2})(1-q^{-3})(1-q^{-4})}.\]

\begin{enumerate}
\item
The expansions
\[\frac1{1-qq_i} = \frac1{1+q_i}+\frac{q_i}{(1+q_i)^2}(q+1)+O((q+1)^2)\]
together with the partial fraction decomposition
\[\dfrac1{q(1-q^{-2})(1-q^{-3})(1-q^{-4})} = -\frac1{16}\frac1{(1+q)^2}+\frac9{32}\frac1{1+q}+O((q+1)^0)\]
writes the residue at $-1$ as
\[-\frac1{16}\frac1{(1+q_1)(1+q_2)}\left(\frac{q_1}{1+q_1}+\frac{q_2}{1+q_2}\right) + \frac9{32}\frac1{(1+q_1)(1+q_2)}.\]
\item
To compute the residue at $\pm i$, observe that the relevant term in the partial fraction decomposition is
\[\frac{q-1}{8(q^2+1)} = \frac{1+i}{16(q-i)}+\frac{1-i}{16(q+i)}.\]
Thus, the contribution is
\[\frac{1+i}{16}\frac1{(1-iq_1)(1-iq_2)}+\frac{1-i}{16}\frac1{(1+iq_1)(1+iq_2)},\]
which simplifies as
\[\frac{1-q_1-q_2-q_1q_2}{8(1+q_1^2)(1+q_2^2)}.\]
\item
To compute the residue at $\omega^{\pm2}$, observe that the relevant term in the partial fraction decomposition is
\[\frac{2q+1}{9(q^2+q+1)} = \frac1{9(q-\omega^2)}+\frac1{9(q-\omega^{-2})}.\]
Thus, the contribution is
\[\frac19\frac1{(1-\omega^2 q_1)(1-\omega^2 q_2)}+\frac19\frac1{(1-\omega^{-2} q_1)(1-\omega^{-2}q_2)},\]
which simplifies as
\[\frac{2+q_1+q_2-q_1q_2}{9(1+q_1+q_1^2)(1+q_2+q_2^2)}.\]
\end{enumerate}

\subsubsection{Contributions from $n=1$ term in $F^{tw}_1(\bar{\t}^{fake})$}
\

The invariant $F^{tw}_1(\bar{\t}^{fake})$ is precisely $F^{tw}_1(\bar{\t}^{new})$ evaluated at $q_1=q_2=0$, which is
\[\frac9{32}+\frac18+\frac29.\]

\subsection{Comparing the results}
\

The differences of terms with poles at $-1, \pm i, \omega^{\pm2}$, in \cite{Lee12}'s and our results, are as follows.
\begin{enumerate}
\item
At $-1$ we have
\[\begin{array}{ll}
&\dfrac{3q_1+4}{8(q_1+1)^2}+\dfrac{3q_2+4}{8(q_2+1)^2}+\dfrac18\dfrac{q_1q_2}{(1-q_1^2)(1-q_2^2)}\left(11-\dfrac{2q_1}{1+q_1}-\dfrac{2q_2}{1+q_2}\right)\\
-&\left(-\dfrac1{16}\dfrac1{(1+q_1)(1+q_2)}\bigg(\dfrac{q_1}{1+q_1}+\dfrac{q_2}{1+q_2}\bigg) + \dfrac9{32}\dfrac1{(1+q_1)(1+q_2)}\right)\\
=&\dfrac{23}{32(1-q_1)(1-q_2)}.
\end{array}\]
\item
At $\pm i$ we have
\[\begin{array}{ll}
&-\dfrac{q_1}{4(q_1^2+1)}-\dfrac{q_2}{4(q_2^2+1)}+\dfrac14\dfrac{q_1q_2}{(1-q_1)(1-q_2)}\dfrac{1+(q_1+q_2)-q_1q_2}{(1+q_1^2)(1+q_2^2)}\\
-&\dfrac{1-q_1-q_2-q_1q_2}{8(1+q_1^2)(1+q_2^2)}\\
=&-\dfrac1{8(1-q_1)(1-q_2)}.
\end{array}\]
\item
At $\omega^{\pm 2}$ we have
\[\begin{array}{ll}
&\dfrac1{3(q_1^2+q_1+1)}+\dfrac1{3(q_2^2+q_2+1)}+\dfrac13\dfrac{q_1q_2}{(1-q_1)(1-q_2)}\dfrac{1+2(q_1+q_2)+q_1q_2}{(1+q_1+q_1^2)(1+q_2+q_2^2)}\\
-&\dfrac{2+q_1+q_2-q_1q_2}{8(1+q_1+q_1^2)(1+q_2+q_2^2)}\\
=&\dfrac4{9(1-q_1)(1-q_2)}.
\end{array}\]
\end{enumerate}

Thus after multiplying $(1-q_1)(1-q_2)$, the result in \cite{Lee12} minus our result gives the vanishing quantity
\[\begin{array}{ll}
&\dfrac{(a_1+1)^2+(a_2+1)^2}{24}+\dfrac{(a_1+1)+(a_2+1)}8-1+\dfrac{a_1a_2}{24}\\
+&\dfrac{23}{32}-\dfrac18+\dfrac49\\
-&\left((a_1+a_2+1)+\dfrac1{24}(a_1^2+a_1a_2+a_2^2)-\dfrac{19}{24}(a_1+a_2)-(\dfrac9{32}+\dfrac18+\dfrac29)\right).
\end{array}\]

\section{$1$-pointed $g=1$ invariants of the point target space}
\

In this section, we use 
\[D=\sum_n \frac{q^n}{(1-q)^{n+1}} \frac{\partial}{\partial t_n}.\]
In particular, we have 
\[D\mathcal{F}_1(\t) = \sum_{d,n} \frac{Q^d}{n!} \<\frac{1}{1-qL}, \t(L), \cdots, \t(L) \>_{1,1+n,d}.\]
Recall that for the point target space
\[J(\tau) = (1-q) e^{\frac{\tau}{1-q}}.\]
So
\[\begin{array}{ll}
\bar{\t}^{new} (x) &= (\bar{\t} (x)+1-x) \cdot e^{\tau/(1-x)} -(1-x),\\
\bar{\t}^{fake} (x) &= (1-x) \cdot e^{\tau/(1-x)} -(1-x).
\end{array}\]

We prove
\begin{thm}
The $g=1$ invariants $D\mathcal{F}_1(\t)$ with $1$ distinguished marked point of the point target space is written in terms of $\<\!\<\frac1{1-qL}\>\!\>_{1,1}$ as
\[\begin{array}{ll}
&\<\!\<\dfrac1{1-qL}\>\!\>_{1,1} + \dfrac1{24}(\dfrac{\bar{t}_2}{1-\bar{t}_1} + \dfrac{q}{1-q})D\tau\\
+&\dsum_{a=0,1,\infty,q}Res_a \left(\dsum_{n=0}^3 \frac1{n!} \<\dfrac{D\bar{\t}^{new} (x^{-1})}{1-xL},\bar{\t}^{new}(x^{-1}),\cdots,\bar{\t}^{new}(x^{-1})\>_{1,n+1} \dfrac{dx}{x}\right)\\
-&\dsum_{a=0,1,\infty,q}Res_a \left(\dsum_{n=0}^3 \frac1{n!} \<\dfrac{D\bar{\t}^{fake} (x^{-1})}{1-xL},\bar{\t}^{fake}(x^{-1}),\cdots,\bar{\t}^{fake}(x^{-1})\>_{1,n+1} \dfrac{dx}{x}\right)
\end{array}\]
\end{thm}

The differentials $D\tau, D\bar{\t}^{new}$ and $D\bar{\t}^{fake}$ are given as follows.

\begin{lemma}
We have
\[\begin{array}{ll}
D\tau&=\dfrac1{(1-\bar{t}_1)(1-q)} e^{\frac{q\tau}{1-q}}\\
D\bar{\t}^{new} (x)&=e^{\frac{\tau}{1-x}} \cdot (\dfrac1{(1-\bar{t}_1)(1-q)} + \dfrac1{1-qx})e^{\frac{q\tau}{1-q}}\\
D\bar{\t}^{fake} (x)&= e^{\frac{\tau}{1-x}} \cdot \dfrac1{(1-\bar{t}_1)(1-q)} e^{\frac{q\tau}{1-q}}
\end{array}\]
\end{lemma}

\begin{proof}
From $\tau = \sum_n t_n \frac{\tau^n}{n!}$ we know that
\[D\tau = \sum_n t_n \frac{\tau^{n-1}}{(n-1)!} D\tau + \sum_n \frac{q^n}{(1-q)^{n+1}}\frac{\tau^n}{n!}.\]
Notice that $\bar{t}_1 = \sum_n t_n \frac{\tau^{n-1}}{(n-1)!}$, so we have
\[D\tau = \dfrac1{(1-\bar{t}_1)(1-q)} e^{\frac{q\tau}{1-q}}.\]
As a direct corollary, applying $D$ to $\bar{\t}^{fake} = (1-x)(e^{\frac{\tau}{1-x}}-1)$, we get
\[D\bar{\t}^{fake} = e^{\frac{\tau}{1-x}} D\tau = e^{\frac{\tau}{1-x}} \cdot \dfrac1{(1-\bar{t}_1)(1-q)} e^{\frac{q\tau}{1-q}}.\]

Next, differentiating $\bar{t}_m = \sum_n t_n \frac{\tau^{n-m}}{(n-m)!}$ with $D$ we get
\[D\bar{t}_m=(\bar{t}_{m+1} + (\frac{q}{1-q})^m (1-\bar{t}_1))D\tau.\]
So
\[D\bar{\t} = \sum_n D\bar{t}_n (x-1)^n = (\frac1{x-1} \bar{\t}+\frac{1-\bar{t}_1}{1-\frac{q}{1-q}(x-1)})D\tau.\]
Thus differentiating $\bar{\t}^{new} = e^{\frac{\tau}{1-x}}(\bar{\t}(x)+1-x)-(1-x)$ we get
\[D\bar{\t}^{new} = e^{\frac{\tau}{1-x}} \cdot (\frac1{(1-\bar{t}_1)(1-q)} + \frac1{1-qx})e^{\frac{q\tau}{1-x}}.\]
\end{proof}

\begin{proof} (of Theorem 2)
We compute $D\mathcal{F}_1(\t)$ with Theorem 1.
The first term simplifies as $\<\!\<\frac1{1-qL}\>\!\>_{1,1}$.

With the lemma, we know that the second term
\[D(\frac1{24} \log \frac{\partial \tau}{\partial t_0}) = -\frac1{24} D\log (1-\bar{t}_1) = \frac1{24}(\frac{\bar{t}_2}{1-\bar{t}_1} + \frac{x}{1-x})D\tau.\]
The third and fourth terms contribute
\[\begin{array}{ll}
&\bigg(Res_0+Res_1+Res_\infty+Res_q\bigg) \left(\dsum_{n=0}^3 \frac1{n!} \<\dfrac{D\bar{\t}^{new} (x^{-1})}{1-xL},\bar{\t}^{new}(x^{-1}),\cdots,\bar{\t}^{new}(x^{-1})\>_{1,n+1} \dfrac{dx}{x}\right)\\
-&\bigg(Res_0+Res_1+Res_\infty+Res_q\bigg) \left(\dsum_{n=0}^3 \frac1{n!} \<\dfrac{D\bar{\t}^{fake} (x^{-1})}{1-xL},\bar{\t}^{fake}(x^{-1}),\cdots,\bar{\t}^{fake}(x^{-1})\>_{1,n+1} \dfrac{dx}{x}\right)
\end{array}\]
by the product rule of derivatives.
\end{proof}

The correlators $\<\!\<\frac1{1-qL}\>\!\>_{1,1}$ is reconstructed from $\<\frac1{1-qL}\>_{1,1}$ by

\begin{prop}
\[\<\!\<\frac1{1-qL}\>\!\>_{0,1} = \sum_{a=0,\infty,q}Res_a \frac{e^{(1-x) \tau} e^{\frac{q\tau}{1-q}}}{1-qx^{-1}} \<\frac1{1-xL}\>_{1,1} \frac{dx}x.\]
\end{prop}

\begin{proof}

We have by \cite{Lee12} that
\[\chi(\M_{1,1}, \frac1{1-x\mathcal{H}^\ast} \cdot \frac1{1-qL})\]
decomposes as partial fractions
\[\begin{array}{ll}
&\dfrac{5q-6}{24(q-1)^2(x-1)} + \dfrac{5q+6}{24(q+1)^2(x+1)} +\dfrac1{24(q-1)(x-1)^2}+\dfrac1{24(q+1)(x+1)^2}\\
+&\dfrac{qx+1}{4(q^2+1)(x^2+1)} + \dfrac{qx-q+1}{6(q^2-q+1)(x^2-x+1)}+ \dfrac{qx+q+1}{6(q^2+q+1)(x^2+x+1)}\end{array}\]
Here, for the last three terms, we have the decompositions
\[\begin{array}{ll}
\dfrac{x}{x^2+1} &= \dfrac12 \left(\dfrac{-i}{1-i x} + \dfrac{i}{1+i x}\right)\\
\dfrac{1}{x^2+1} &= \dfrac12 \left(\dfrac1{1-i x} + \dfrac1{1+i x}\right)\\
\dfrac{x-1}{x^2-x+1} &= - \dfrac13 \left(\dfrac{1+\omega}{1-\omega x} + \dfrac{1+\omega^{-1}}{1-\omega^{-1} x}\right)\\
\dfrac1{x^2-x+1} &= \dfrac13 \left(\dfrac{2-\omega}{1-\omega x} + \dfrac{2-\omega^{-1}}{1-\omega^{-1} x}\right)\\
\dfrac{x+1}{x^2+x+1} &= \dfrac13 \left(\dfrac{1-\omega^2}{1-\omega^2 x} + \dfrac{1-\omega^{-2}}{1-\omega^{-2} x}\right)\\
\dfrac{1}{x^2+x+1} &= \dfrac13 \left(\dfrac{2+\omega^2}{1-\omega^2 x} + \dfrac{2+\omega^{-2}}{1-\omega^{-2} x}\right)
\end{array}\]

On the other hand, by String equation \cite{Lee01} Section 4.4, we have
\[\<\!\<\frac1{1-qL}\>\!\>_{0,1} = \sum_{n \geq 0} \frac{\tau^n}{n!} \chi(\M_{1,1},(1-\mathcal{H}^\ast+\frac{q}{1-q})^n\cdot \frac1{1-qL}).\]
We know that if $\sum_n a_nx^n = \frac1{1-x\zeta}$, then $a_n = \zeta^n$ and with $(\mathcal{H}^\ast)^n$ substitited by $a_n$ we have
\[\sum_n \frac{\tau^n}{n!} (1-\mathcal{H}^\ast+\frac{q}{1-q})^n = \exp (\frac{\tau}{1-q} - \tau \zeta).\]
If $\sum_n a_nx^n = \frac1{(1-x\zeta)^2}$, then $a_n = (n+1)\zeta^n$ and with $(\mathcal{H}^\ast)^n$ substitited by $a_n$ we have
\[\sum_n \frac{\tau^n}{n!} (1-\mathcal{H}^\ast+\frac{q}{1-q})^n = (1-\zeta\tau)\exp (\frac{\tau}{1-q} - \tau \zeta).\]
Thus $\<\!\<\frac1{1-qL}\>\!\>_{0,1}$ is written as substituting $\frac1{1-x\zeta}$ by $\exp (\frac{\tau}{1-q} - \tau \zeta)$ and $\frac1{(1-x\zeta)^2}$ by $(1-\zeta\tau)\exp (\frac{\tau}{1-q} - \tau \zeta)$ in $\chi(\M_{1,1}, \frac1{1-x\mathcal{H}^\ast} \cdot \frac1{1-qL})$.

Putting everything together, the total contribution is
\[\begin{array}{ll}
&-\dfrac{5q-6}{24(q-1)^2} \cdot e^{-\tau}e^{\frac{\tau}{1-q}} + \dfrac{5q+6}{24(q+1)^2} \cdot e^{\tau}e^{\frac{\tau}{1-q}} \\
&+\dfrac1{24(q-1)} \cdot (1-\tau) e^{-\tau}e^{\frac{\tau}{1-q}} - \dfrac1{24(q+1)} \cdot (1+\tau) e^{\tau}e^{\frac{\tau}{1-q}} \\
&+\dfrac1{8(q^2+1)} \big((-iq+1)e^{-i\tau}+(iq+1)e^{i\tau}\big)e^{\frac{\tau}{1-q}}\\
&+\dfrac1{18(q^2-q+1)} \big((-(1+\omega)q+(2-\omega))e^{-\omega\tau}+(-(1+\omega^{-1})q+(2-\omega^{-1}))e^{-\omega^{-1}\tau}\big)e^{\frac{\tau}{1-q}}\\
&+\dfrac1{18(q^2+q+1)} \big(((1-\omega^2)q+(2+\omega^2))e^{-\omega\tau''}+((1-\omega^{-2})q+(2+\omega^{-2}))e^{-\omega^{-2}\tau}\big)e^{\frac{\tau}{1-q}}
\end{array}\]
which is simplified as
\[e^{\frac{\tau}{1-q}} \cdot \left( \dsum_{\zeta = \pm1}\big(\dfrac{\tau}{24(1-\zeta q)} + \dfrac{5-4\zeta q}{24(1-\zeta q)^2}\big)e^{\zeta\tau} + \dsum_{\zeta = \pm i} \dfrac1{4(1-\zeta^2)} \dfrac{e^{\zeta^{-1}\tau}}{1-\zeta q} + \dsum_{\zeta = \omega^{\pm1,\pm2}} \frac1{6(1-\zeta^2)}\dfrac{e^{\zeta^{-1}\tau}}{1-\zeta q} \right).\]

The right-hand side residue in our proposition is written as
\[\sum_{\zeta = \pm1,\pm i,\omega^{\pm1,\pm2}} Res_\zeta \frac{e^{(1-x^{-1}) \tau} e^{\frac{q\tau}{1-q}}}{1-qx} \<\frac1{1-x^{-1}L}\>_{1,1} \frac{dx}x\]
by the residue theorem and the substitution $x \to x^{-1}$.
Computing this residue with the partial fraction decomposition
\[\begin{array}{ll}
&\<\dfrac1{1-xL}\>_{1,1} = \dfrac1{(1-x^4)(1-x^6)}\\
=&\dsum_{\zeta = \pm1} \dfrac{5-4\zeta x}{24(1-\zeta x)^2} +\dsum_{\zeta = \pm i} \dfrac1{4(1-\zeta^2)} \dfrac1{1-\zeta x}+\dsum_{\zeta = \omega^{\pm1,\pm2}} \dfrac1{6(1-\zeta^2)} \dfrac1{1-\zeta x}.
\end{array}\]
proves the proposition.
\end{proof}

\section*{Appendix: A recursive formula for $\tau$}
\

In this appendix, we assume that $\Lambda$ is a local ring with maximal ideal $\Lambda_+$ that contain $t_n^\alpha$'s and $Q$'s, where $t_n^\alpha$'s are coefficients in $\t = \sum_{n,\alpha} t_n^\alpha \phi_\alpha q^n$.
Assume also that $\cap_n \Lambda_+^n = \{0\}$.
Recall that $K=K^0(X) \otimes \Lambda$.
We define a map $T: K \to K$, by mapping $\tau$ to
\[T(\tau) = \t(1) + \left[\dsum_{\alpha,\beta} \<\!\<LD\t(L), \phi_\alpha\>\!\>_{0,2}G^{\alpha\beta}\phi_\beta\right],\]
where $D\t(L)=\frac{\t(1)-\t(L)}{1-L}$.

We prove the following proposition.
\begin{prop*}
The sequence $\tau_n = T^n(0)$ converges to (the unique) $\tau \in K$ that makes $\bar{\t}(1)=0$. 
Moreover, $\tau_n-\tau \in \Lambda_+^n$.
\end{prop*}

The map $T$ is a contraction mapping modulo powers of the maximal ideal $\Lambda_+$:

\begin{lemma*}
If $\tau' - \tau'' \in \Lambda_+^n$, then $T(\tau')-T(\tau'') \in \Lambda_+^{n+1}$.
\end{lemma*}

\begin{proof}
From $\tau' \equiv \tau''$ modulo $\Lambda_+^n$, we know that
\[\dsum_{\alpha,\beta,\gamma} \<\!\<\dfrac{\phi^\gamma}{1-qL}, \phi_\alpha\>\!\>_{0,2}\phi_\gamma \otimes G^{\alpha\beta}\phi_\beta\]
with inputs $\tau', \tau''$ for the $\tau$-part are congruent modulo $\Lambda_+^n$.
Now it follows from $\mathfrak{t}(q) = \t(q)-\t(1)+D\t(q) \in \Lambda_+$, and 
\[\dsum_{\alpha,\beta} \<\!\<\mathfrak{t}, \phi_\alpha\>\!\>_{0,2}G^{\alpha\beta}\phi_\beta = \Omega \left(\dsum_{\alpha,\beta,\gamma} \<\!\<\dfrac{\phi^\gamma}{1-qL}, \phi_\alpha\>\!\>_{0,2}\phi_\gamma \otimes G^{\alpha\beta}\phi_\beta, \mathfrak{t}(q)\right)\]
that $T(\tau')-T(\tau'') \in \Lambda_+^{n+1}$.
Here $\Omega$ is the pairing on $K \otimes \mathbb{Q}((q))$ given by $\Omega(f,g) = (Res_0+Res_\infty) (f(q^{-1}),g(q))\frac{dq}q$ \cite{perm7}.
\end{proof}

As a direct corollary, the fixed point of $T$ is unique.
Observe that $\tau_n = T^n(0)$ satisfy $\tau_0=0$ and $\tau_1=\t(1) \in \Lambda_+$.
So by Lemma 1 we know that $\tau_{n-1}-\tau_n \in \Lambda_+^n$, and thus $\{\tau_n\}$ converges to some fixed point $\tau \in K$ of $T$.

\begin{lemma*}
$\bar{\mathbf{t}}(1)=0$ if and only if $\tau$ is a fixed point of $T$.
\end{lemma*}

\begin{proof}
By using the fact
\[[f(q)]_+ = -(Res_{w=0}+Res_{w=\infty}) \frac{f(w)}{w-q}dw,\]
we have
\[[\frac{\t(q)+1-q}{1-L/q}]_+ = \frac{q(\t(q)+1-q)-L(\t(L)+1-L)}{q-L}.\]
Evaluating at $q=1$, we have $(\t+D\t)(L) - L$.

So $\bar{\mathbf{t}}_1(1) = [S_1(\t(q)+1-q)]_+(1)$ is
\[\sum_{\alpha,\beta} \phi_\alpha G^{\alpha\beta}((\phi_\beta,\t(1))+\<\!\<\phi_\beta,\t(L)+D\t(L)-L\>\!\>_{0,2}).\]
By string and dilaton equations, 
\[\<\!\<L,\phi_\alpha\>\!\>_{0,2}=(\tau,\phi_\alpha) + \<\!\<\tau,\phi_\alpha\>\!\>_{0,2},\]
and so
\[\begin{array}{ll}
\bar{\mathbf{t}}(1) & =\dsum_{\alpha,\beta} \phi_\alpha G^{\alpha\beta}\left[(\phi_\beta,\t(1)-\tau) +\<\!\<\phi_\beta,\t(L)+D\t(L)-\tau\>\!\>_{0,2}\right]\\
& = \t(1) -\tau + \dsum_{\alpha,\beta} \<\!\<LD\t(L), \phi_\alpha\>\!\>_{0,2}G^{\alpha\beta}\phi_\beta\\
&= T(\tau)-\tau.
\end{array}\]
\end{proof}

\author{Dun Tang, Department of Mathematics, University of California, Berkeley, 
1006 Evans Hall, University Drive, Berkeley CA94720.
E-mail address: dun\_tang@math.berkeley.edu}


\begin{thebibliography}{3}

\bibitem{Chou23} Y.C. Chou, L. Herr, Y.P. Lee. Higher genus quantum K-theory. {\it Pacific J. Math.} 330: 85-121, 2024.

\bibitem{WDVV} A.B. Givental. On the WDVV equation in quantum K-theory. {\it Michigan Math. J.} 48 (1): 295 - 304, 2000.

\bibitem{perm7} A.B. Givental. Permutation-equivariant quantum K-theory VII. General theory. {\it arXiv:1510.03076 [math.AG]}, 2015.

\bibitem{perm10} A.B. Givental. Permutation-equivariant quantum K-theory X. Quantum Hirzebruch-Riemann-Roch in genus 0. {\it SIGMA} 031, 2020.

\bibitem{Lee01} Y.-P. Lee. Quantum K-theory, I: Foundations. {\it Duke Math. J.} 121 (3): 389 - 424, 2004.

\bibitem{Lee12} Y.-P. Lee, F. Qu. Euler characteristics of universal cotangent line bundles on $\M_{1,n}$. {\it arXiv:1211.2450v1 [math.AG]}, 2012.

\bibitem{Tonita14} V. Tonita. A virtual Kawasaki Riemann–Roch formula. {\it Pacific J. Math.} 268 no. 1: 249–255, 2014.

\end{thebibliography}
\end{document}